\documentclass[11pt]{article}

\usepackage{ulem}
\usepackage{amssymb,amsbsy,amsmath,amsfonts,amscd,amsthm}
\usepackage{xcolor}
\usepackage{bbm}
\usepackage{graphicx} 
\usepackage{multicol}
\graphicspath{{./}{immagini/}} 
\usepackage{tikz}
\usetikzlibrary{matrix}

\newcommand{\vl}{v^\lambda}
\newcommand{\wl}{w^\lambda}
\newcommand{\ds}{\displaystyle}
\newcommand{\med}[1]{\langle #1\rangle}
\newcommand{\R}{\mathbb{R}}

\newcommand{\Z}{\mathbb{Z}}
\newcommand{\N}{\mathbb{N}}

\newcommand{\EE}{\mathbb{E}}
\newcommand{\ZZ}{\mathbb{Z}}
\newcommand{\mS}{\mathbb{S}}
\newcommand{\mL}{\mathcal{L}}
\newcommand{\mA}{\mathcal A}

\newcommand{\mM}{\mathcal M}

\newcommand{\mT}{\mathcal T}
\newcommand{\mX}{\mathcal X}
\newcommand{\mV}{\mathcal V}

\def\a{\alpha}
\def\b{\beta}
\def\d{\delta}

\def\g{\gamma}
\def\l{\lambda}

\def\m{\mu}
\def\n{\nu}
\def\r{\rho}
\def\s{\sigma}

\def\G{\Gamma}

\def\th{\theta}

\def\t{\tau}
\def\e{\varepsilon}

\def\pd{\partial}
\def\half{\frac{1}{2}}

\newcommand{\cA}{{\cal A}}

\newcommand{\cE}{{\cal E}}

\newcommand{\cM}{{\cal M}}
\newcommand{\cL}{{\cal L}}

\newcommand{\cV}{{\cal V}}

\newcommand{\bm}{\bar m}
\newcommand{\hx}{{\hat x}}
\newcommand{\tx}{{\tilde{x}}}
\newcommand{\tphi}{{\tilde{\phi}}}

\newcommand{\bC}{{\bar{C}}}
\newcommand{\ue}{u^{\varepsilon}}

\DeclareMathOperator{\Wass}{\mathbf{d}}
\newtheorem{theorem}{Theorem}[section]
\newtheorem{lemma}[theorem]{Lemma}
\newtheorem{definition}[theorem]{Definition}
\newtheorem{proposition}[theorem]{Proposition}

\newtheorem{remark}[theorem]{Remark}

\numberwithin{equation}{section}

\begin{document}
\title{A continuous dependence estimate for viscous Hamilton-Jacobi equations on networks with applications}
\date{version: \today} 
\author{Fabio Camilli\footnotemark[1]  \and  Claudio Marchi\footnotemark[2]}
\maketitle
\footnotetext[1]{Dip. di Scienze di Base e Applicate per l'Ingegneria,  Universit{\`a}  di Roma  ``La Sapienza", via Scarpa 16, 00161 Roma, Italy  ({\tt  fabio.camilli@uniroma1.it}).}
\footnotetext[2]{Dip. di Matematica ``Tullio Levi-Civita'', Universit\`a di Padova, via Trieste, 35121 Padova, Italy ({\tt claudio.marchi@unipd.it}).}

\begin{abstract}
We study continuous dependence estimates for viscous Hamilton-Jacobi  equations  defined on a network $\G$.  Given two Hamilton-Jacobi equations, we prove an estimate of the $C^2$-norm of the difference between  the  corresponding solutions in terms of the   distance among the   coefficients. We also provide two applications of the previous estimate: the first one is an existence and uniqueness result for a quasi-stationary Mean Field Games defined on the network $\G$; the second one is an estimate of the rate of convergence for homogenization of Hamilton-Jacobi equations defined on a periodic network, when the size of the cells vanishes and the limit problem is defined in the whole Euclidean space.
\end{abstract}
\noindent
{\footnotesize \textbf{AMS-Subject Classification:} 35R02, 49N70, 91A16, 35B27}.

\section{Introduction}
In the recent years, there is an increasing interest in the study of  dynamical system on networks,  in connection with problem such as vehicular traffic, data transmission, crowd motion, supply chains, etc. As consequence, many results for linear and nonlinear PDEs  in the Euclidean  case  have been progressively  extended to the network setting and also to more general geometric structures. Here, we are interested in   continuous dependence estimates for viscous Hamilton-Jacobi (HJ for short) equations. Let us recall that such estimates play a crucial  role  in many contexts, for example for regularity results, error estimate for numerical schemes, rate of of convergence in vanishing viscosity and homogenization \cite{cdi,jk}.  \\
Our analysis is inspired by the results in \cite{marchi}, where it is proved a continuous dependence estimate in  the $C^2$-norm for  solutions of a viscous HJ   equations in the periodic setting with an explicit dependence on the distance of the coefficients and an explicit characterization of the constants. We prove an analogous result for viscous HJ equations defined on networks with Kirchhoff conditions at the vertices. To this end, we use some results concerning the study of these equations on networks \cite{adlt1,adlt2,cms} and suitably adapt the   arguments  in \cite{marchi} to this specific setting. \\
Then, the previous continuous dependence estimate   is applied to two problems:
\begin{description}
\item[(i)]  the well-posedness of a quasi-stationary Mean Field Games system defined  on a network;
\item[(ii)] an estimate  of the rate of convergence for    homogenization of HJ equations defined on a periodic networks.
\end{description}
Mean Field Games (MFG for short), introduced in \cite{ll}, modelize the interaction among a large number of agents. In this theory, the agents are assumed indistinguishable, infinitesimal and completely rational and their behaviour is influenced by the statistical distribution of the states of the other agents. In the classical formulation, MFG lead to the study of a coupled system of two evolutive PDEs, a backward HJ equation for the value function of the representative agent, a forward Fokker-Planck (FP for short) equation for the distribution of the agents. Recently, a different strategy mechanisms from classical MFG theory has been proposed in \cite{mouzouni} (see also \cite{cristiani_et_altri}): the agents are myopic and choose their strategy only according to the information available at present time, without forecasting the future evolution. In this case, the Nash equilibria for the distribution of the agents are characterized  by a quasi-stationary MFG system, which is composed of a stationary HJ equation and a evolutive Fokker-Planck equation.\\ 
While classical MFG on networks have been studied in \cite{adlt1, adlt2,cm}, here we consider a  quasi-stationary MFG defined  on a network and we prove existence and uniqueness of the corresponding solution. Existence is proved via a fixed point argument and the continuous dependence estimate is crucial since in this case it is not possible to exploit the regularizing effect of the parabolic HJ equation to  show the continuity of the fixed point map. The continuous dependence estimate is also exploited to prove uniqueness of the solution, which, with respect to the classical case, requires no monotonicity assumption.\\
The second application of the continuous dependence estimate is to a  homogenization problem. By means of the classical perturbed test function method (see \cite{evans}), we show that the solution of a viscous HJ equation, defined on a periodic lattice of size $\epsilon$, converges, as $\epsilon\to 0$, to the solution of an effective   problem defined in all the Euclidean space and we also give an estimate of the rate of convergence. Moreover, we obtain   a characterization of the corresponding effective operator in terms of the Hamiltonians defined on the edges of the lattice. We note that a similar problem  was studied  for first order HJ equations in \cite{im} and for linear second order equations in \cite{bnv}.\par
The paper is organized as follows: in Sect.~\ref{sect:definitions} we fix our setting and our notations for the network. Sect.~\ref{sec:continuous_dep_est} is devoted to our main result, the continuous dependence estimate for the solution to an HJ equation on the network. In Sect.~\ref{sec:quasi_stat} we tackle quasi-stationary MFGs on the network: in particular, we obtain existence and uniqueness of a solution without requiring any monotonicity assumption. Sect.~\ref{sec:homog} concerns the homogenization of HJ equations on a lattice: the main result is a rate of convergence estimate.
\section{The network $\G$: notations and definitions}\label{sect:definitions}
We consider a bounded network $\Gamma\subset\R ^N$  composed by  a  finite collection of   bounded   straight  edges $\mathcal{E}:=\left\{ \Gamma_{\alpha}, \alpha\in\mA\right\}$, which connect   a finite collection of vertices $\mV:=\left\{ \nu_{i}, i\in I\right\}$.  We assume that, for $\alpha,\beta \in \cA$ with $\alpha\not=\beta$,  $\Gamma_\alpha\cap \Gamma_\beta$ is either empty or made of  a single vertex. For an edge $\Gamma_{\alpha}\in\mathcal{E}$ connecting two vertices $\n_i$ and $\n_j$ with $i<j$, we consider the parametrization $\pi_\alpha: [0,\ell_\alpha]\to \Gamma_\alpha$ given by
\begin{equation*}
\pi_\alpha(y)=[y\nu_j+(\ell_\alpha-y)\nu_i]\ell_a^{-1}\quad\text{for } y\in  [0,\ell_\alpha],
\end{equation*}
where $\ell_\alpha$ is the length of the edge.  We also denote with $\mA_{i}=\left\{ \alpha\in\mA:\nu_{i}\in\Gamma_{\alpha}\right\} $  the set of indices of edges that are adjacent to the vertex $\nu_{i}$.\\
For a function  $v:\Gamma\rightarrow\mathbb{R}$,
we denote with $v_\alpha: (0,\ell_\alpha)\rightarrow\mathbb{R}$  the restriction of $v$ to $\Gamma_{\alpha}$, i.e.
\[
v_{\alpha} (y):=v|_{\Gamma_{\alpha}}\circ\pi_{\alpha} (y),\quad \hbox{ for all }y\in (0,\ell_\alpha).
\]
Moreover we define for $x\in \Gamma_{\alpha}\backslash{\mV}$ the derivative along the arc
\begin{equation*}
\partial_\alpha v(x)=\frac{d v_{\alpha}}{dy}(y)\quad \text{for  $y=\pi_{\alpha}^{-1} (x)$}.\label{eq: derivative}
\end{equation*}
\begin{remark}\label{rmk:extension}
The function~$v_\alpha$ is defined only on~$(0,\ell_\alpha)$; nevertheless, when it is possible, we denote $v_\alpha$ also its extension by continuity on~$0$ and on~$\ell_\alpha$. Note that, in this way, $v_\alpha$ may not coincide with the original function~$v$ at the vertices when $v$ is not continuous.
\end{remark}
For $x=\nu_i\in \Gamma_{\alpha}$, we define the outward derivative at the vertex
\begin{equation*}
\partial_{\alpha}v\left(\pi^{-1}_\alpha (\nu_i)\right):=\begin{cases}
{\displaystyle \lim_{h\rightarrow0^{+}}
\dfrac{v_{\alpha}(0)-v_{\alpha}(h)}{h}}, & \text{if }\nu_i=\pi_{\alpha}\left(0\right),\\[4pt]
{\displaystyle \lim_{h\rightarrow0^{+}}
\dfrac{v_{\alpha}(\ell_{\alpha})-v_{\alpha}(\ell_{\alpha}-h)}{h}}, & \text{if }\nu_i=\pi_{\alpha}\left(\ell_{\alpha}\right).
\end{cases}\label{eq: inward derivative}
\end{equation*}
Setting
\begin{equation*}
n_{i\alpha}=\left\{
\begin{array}[c]{rl}
1 & \text{if }   \nu_i =    \pi_{\alpha} (\ell_\alpha),\\
-1 & \text{if }  \nu_i =    \pi_{\alpha} (0),   
\end{array}\right.
\end{equation*}
we have
\begin{equation*}
\partial_\alpha  v(\nu_i)= n_{i\alpha} \, \partial v_ \alpha (\pi^{-1}_\alpha (\nu_i)) .
\end{equation*}

We introduce some functional spaces defined on the network $\Gamma$. The space
$C(\Gamma)$ is composed of the continuous functions on $\Gamma$; the space
\[
PC\left(\Gamma\right):=\left\{ v : \Gamma \to \R \;:  \hbox{$v_\alpha\in C([0,\ell_\alpha])$,  for all $\alpha\in\mA$}\right\}
\]
is composed of the piece-wise continuous functions on $\Gamma$, i.e. functions which are continuous inside the edges but not necessarily at the vertices.  
For $m\in\mathbb{N}$
\[
C^{m}\left(\Gamma\right):=\left\{ v\in C\left(\Gamma\right):v_{\alpha}\in C^{m}\left(\left[0,\ell_{\alpha}\right]\right)\text{ for all }\alpha\in\mA\right\} ,
\]
is the space of $m$-times continuously differentiable functions on $\Gamma$  endowed with the norm
$$ 
\left\Vert v\right\Vert _{C^{m}\left(\Gamma\right)}:= {\sum}_{\alpha\in\mA}{\sum}_{k\le m}\left\Vert \partial^{k}v_{\alpha}\right\Vert _{L^{\infty}\left(0,\ell_{\alpha}\right)}.
$$
For $\sigma\in\left(0,1\right]$,  the space $C^{m,\sigma}\left(\Gamma\right)$
contains the functions
$v\in C^{m}\left(\Gamma\right)$ such that $\partial^{m}v_{\alpha}\in C^{0,\sigma}\left(\left[0,\ell_{\alpha}\right]\right)$
for all $\alpha\in \mA$ with the norm
$$\displaystyle{
\left\Vert v\right\Vert _{C^{m,\sigma}\left(\Gamma\right)}:=\left\Vert v\right\Vert _{C^{m}\left(\Gamma\right)}+\sup_{\alpha\in\mA}} [\partial^{m}v_{\alpha}]_{\s, [0,\ell_\alpha]}
$$
where, for $\s \in (0,1]$ and $w:A\to\R$,
$$
[w]_{\s,A}=\sup_{y\ne z \atop y,z\in  A}\dfrac{\left|w\left(y\right)-w\left(z\right)\right|}{\left|y-z\right|^{\sigma}}.
$$
The integral of a function $v$ on $\Gamma$ is defined by
$$\int_{\Gamma}v(x)dx=\sum_{\alpha\in\mA}\int_{0}^{\ell_{\alpha}}v_{\alpha}\left(y\right)dy
$$
and we set $\med{v}=\int_{\Gamma}v(x)dx$. For $p\in [1,\infty]$, we  define the Lebesgue space  
$$L^{p}\left(\Gamma\right)  =\left\{ v:v_\a\in L^{p}\left((0,\ell_\alpha)\right)\text{ for all \ensuremath{\alpha\in\mA}}\right\},$$
endowed with the standard norm. For any integer $m\in\N$, $m\ge 1$, and   $p\in [1,\infty]$ 
we define the Sobolev space 
$$
W^{m,p}(\Gamma):=\left\{ v\in C\left(\Gamma\right):v_{\alpha}\in W^{m,p}\left((0,\ell_{\alpha})\right)\text{ for all }\alpha\in\mA\right\},
$$
endowed with the norm
$$
\left\Vert v\right\Vert _{W^{m,p}\left(\Gamma\right)}=\left(\sum^{m}_{k=1}\sum_{\alpha\in\mA}
\left\Vert \partial^{k}v_{\alpha}\right\Vert _{L^{p}\left(0,\ell_{\alpha}\right)}^{p}+
\left\Vert v\right\Vert _{L^p(\Gamma)}^{p}\right)^{\frac 1 p}.
$$
We also set $H^m(\Gamma)= W^{m,2}(\Gamma)$.\\
The couple $(\G,d_\G)$, where $d_\G$  is the geodesic distance    on the network, is a metric space. Denote with $\mM$ the space of  Borel probability measures on $\G$. For $1 \le p < \infty$, the $L^p$-Wasserstein distance $\Wass_p$ between $\s,\t \in\mM$ is defined by the Monge-Kantorovich transport problem
\begin{align*}
\Wass_p(\s,\t)=\min_{\Sigma\in\Pi(\s,\t)}\left\{\int_{\G\times\G}d_\G^p(x,y)d\Sigma(x,y) \right\}
\end{align*}
where  $\Pi(\s,\t)$ denotes the set of transport plans, i.e.   Borel probability measures on $\G\times\G$ with marginals $\s$ and $\tau$ (see \cite{bogachev}). Since $\G$ is compact, the Wasserstein distance $\Wass_p$ metrises the topology of weak convergence of probability measures on $\G$. In particular, for $p=1$, we have
\begin{equation}\label{eq:wass_dist}
\Wass_1(\s,\t)=\sup\left\{\int_\G f(x)d(\s-\t):\, f:\G\to\R,\,|f(x)-f(y)|\le d_\G(x,y)\right\}.
\end{equation}
We shortly recall the definition of diffusion process on the network $\G$ (see \cite{fw, fs} for details). Consider the linear   differential operator $\mL$ defined on the edges  by
\begin{equation*}
\mL_\a u(x)= \mu_{\alpha}\partial^{2}u(x)+B_\a(x)\partial u(x), \quad 
 x\in \Gamma_\alpha,\, \a\in \mA
\end{equation*}
with domain
\begin{equation*}
D\left(\mL\right)=\left\{ u\in C^{2}\left(\Gamma\right):\sum_{\alpha\in\mA_{i}}p_{i,\alpha}\partial_{\alpha}u\left(\nu_{i}\right)=0,\;i\in I \right\} \label{eq:Kir}
\end{equation*}
where $p_{i,\a}\in (0,1)$, $\sum_{\a\in \mA_i} p_{i,\a}=1$. 
Then, the  operator $\cL$ is the infinitesimal generator of a Feller-Markov process $(X_t, \alpha_t)$, with  $X_t\in \Gamma_{\alpha_t}$,  such that, for $x_t= \pi_{\alpha_t}^{-1}(X_t)$, we have
\begin{eqnarray}
\label{eq:diff}
dx_t=   B_{\alpha_t}(x_t) dt+\mu_{\alpha_t} dW_t  + d\ell_{i,t} +d h_{i,t}.
\end{eqnarray}
In \eqref{eq:diff}, $W_t$ is a one dimensional Wiener process; $\ell_{i,t}$ and $h_{i,t}$, $i\in I$, are continuous non-decreasing and, respectively, non-increasing processes,  measurable with respect to the $\sigma$-field generated by $(X_t,\alpha_t)$ and satisfying
\begin{align*}
\hbox{$\ell_{i,t}$ increases only when $X_t=\nu_i$ and $x_t=0$,}   \\
\hbox{$h_{i,t}$ decreases only when $X_t=\nu_i$ and $x_t=1$.} 
\end{align*}

\section{The continuous dependence estimate}\label{sec:continuous_dep_est}
We consider the following HJ equation on $\Gamma$
\begin{equation}
\begin{cases}
-\mu_{\alpha}\partial^{2}v+H\left(x,\partial v\right)+\rho=0, & x\in\left(\Gamma_{\alpha}\backslash\mV\right),\alpha\in\mA,\\
{\displaystyle \sum_{\alpha\in\mA_{i}}\gamma_{i\alpha}\mu_{\alpha}\partial_{\alpha}v(\nu_i)=0,} 
& \nu_{i}\in\mV,\\
v|_{\Gamma_{\alpha}}(\nu_i)=v|_{\Gamma_{\beta}}(\nu_i), & \alpha,\beta\in\mA_{i},\nu_{i}\in\mV,\\ \med{v}=0.
\end{cases}\label{eq:HJ}
\end{equation}
where $H: \Gamma\times\R\to \R$ is given for $x\in\Gamma_\a$ by
\begin{equation}\label{eq:hamiltonian}
H_\a(x,p)=\sup_{a\in A}\left\{ - b_\a(x,a)p-f_\a(x,a)\right\}.
\end{equation} 
Problem \eqref{eq:HJ} represents the dynamic programming equation for the optimal control problem  with long-run average cost functional
\[
\rho=\inf_{a }\,\liminf_{T\to \infty}\frac 1 T \EE_x\left[\int_0^T f(X_t,a_t)dt  \right]
\]
where $a_t$ is a feedback control law of form $a_t=a(X_t)$ and $X_t$  is a diffusion process on $\G$ such that $x_t= \pi_{\alpha_t}(X_t)$ satisfies \eqref{eq:diff} with $B_\a(x)=B_\a(x,a(x))$ (see \cite[Section 1.3]{adlt1} for more details). Connected with the optimal control interpretation of \eqref{eq:HJ}, the second equation is a Kirchhoff   transmission condition, where the quantity 
$p_{i\a}=\gamma_{i\alpha}\mu_{\alpha}(\sum_{\a\in\mA_i} \gamma_{i\alpha}\mu_{\alpha})^{-1}$ represents the probability that the trajectories of the diffusion process enter in edge $\G_\a$, $\a\in\mA_i$, from the vertex $\nu_i$; it can be also interpreted as a Neumann boundary condition if $\sharp(\mA_i)=1$.
 The third equation implies continuity of the solution at the vertices and the last one is a normalization condition.\\
We assume that $A$ is a compact  separable metric space  (for simplicity, $A$ is a subset of some Euclidean space) and we make the following assumptions
\begin{itemize}
\item[(H1)] $\mu_\a$ and $\g_{i,\a}$ are positive constants and
\[\sum_{\alpha\in\mA_{i}}\gamma_{i\alpha}\mu_{\alpha}=1, \qquad  \a\in \mA,\,i\in\mA_i.\]
\item[(H2)] $b_\a,f_\a:\G_\a\times A\to\R$ are continuous and there exist two constants $K$ and $L$ such that
\begin{equation}\label{hyp_b}
\begin{split}
&|b_\a(x,a)|\leq K,\qquad |b_\a(x_1,a)-b_\a(x_2,a)|\leq L|x_1-x_2|\\
&|f_\a(x,a)|\leq K,\qquad |f_\a(x_1,a)-f_\a(x_2,a)|\leq L|x_1-x_2|
\end{split}
\end{equation}
for all $x,x_1,x_2\in \Gamma_\a$, $a\in A$ and $\a\in\mA$.
\end{itemize}
For the study of the ergodic problem \eqref{eq:HJ}, it is expedient to introduce for $\l\in (0,1)$ the  discount   approximation 
\begin{equation}\label{eq:HJ_discount}
\begin{cases}
-\mu_{\alpha}\partial^{2}\vl+H\left(x,\partial \vl\right)+\l\vl=0, & x\in\left(\Gamma_{\alpha}\backslash\mV\right),\alpha\in\mA,\\[4pt]
{\displaystyle \sum_{\alpha\in\mA_{i}}\gamma_{i\alpha}\mu_{\alpha}\partial_{\alpha}\vl(\nu_i)=0,} 
& \nu_{i}\in\mV,\\[4pt]
\vl|_{\Gamma_{\alpha}}(\nu_i)=\vl|_{\Gamma_{\beta}}(\nu_i), & \alpha,\beta\in\mA_{i},\nu_{i}\in\mV.
\end{cases}
\end{equation}
The following statement concerns existence, uniqueness and regularity of classical solutions to the HJ equations \eqref{eq:HJ} and \eqref{eq:HJ_discount} (see \cite[Theorem II.2]{al}, \cite{cms} and \cite[Proposition 3.2 and Theorem 3.7]{adlt1}).
\begin{proposition}\label{lemma:bounds_discount_ergodic}
There exists a unique classical  solution $\vl$ to the equation \eqref{eq:HJ_discount}.
Moreover,
\begin{itemize}
\item[(i)] there exist a positive constant  $C_1$ and $\th\in(0,1)$, both independent of $\l$,  such that
\begin{align}
&\|\l \vl\|_{L^\infty(\G)} \le K,\label{eq:bound_1}\\
&\|\vl-\med{\vl}\|_{C^{2,\theta}(\G)} \le C_1(1+K+L)=:\bar K,\label{eq:bound_2}
\end{align}
where $K$, $L$ as in \eqref{hyp_b};
\item[(ii)] for $\l\to 0^+$, $\l \vl \to \r$, $\vl-\med{\vl}\to v$ and the couple $(v,\r)$ is the unique classical solution to \eqref{eq:HJ}. 		Moreover
\begin{equation}\label{eq:bound_3}
\|v\|_{C^{2,\theta}(\G)} \le \bar K.
\end{equation}
\end{itemize}
\end{proposition}
We give  some preliminary results for equations \eqref{eq:HJ} and \eqref{eq:HJ_discount}. The first result is a    strong maximum principle  for the linear HJ equation (see \cite[Lemma 2.8]{adlt1} or \cite[Theorem 3.1]{cms}).
\begin{lemma}\label{lem:max_princ}
For $g\in PC\left(\Gamma\right)$, the solutions of 
\begin{equation*}
\begin{cases}
-\mu_{\alpha}\partial^{2}v+g\partial v=0, & \text{in }\Gamma_{\alpha}\backslash\mV,\alpha\in\mA,\\
{\displaystyle \sum_{\alpha\in\mA_{i}}\gamma_{i\alpha}\mu_{\alpha}\partial_{\alpha}v(\nu_i)=0,} & i\in I,\\
v|_{\Gamma_\alpha} (\nu_i) =v|_{\Gamma_\beta} (\nu_i), \quad &\alpha,\beta \in \cA_i, i\in I
\end{cases} 
\end{equation*}	
are the constant functions on $\Gamma$.
\end{lemma}
The second result is a comparison principle for \eqref{eq:HJ_discount} (see \cite[Lemma 3.6]{adlt1} and \cite[Corollary 3.1]{cms}).
\begin{lemma}\label{lem:comp_prin} If $u,v\in C^{2}\left(\Gamma\right)$ satisfy
\begin{equation*}
\begin{cases}
-\mu_{\alpha}\partial^{2}v+H\left(x,\partial v\right)+\lambda v\ge-\mu_{\alpha}\partial^{2}u+H\left(x,\partial u\right)+\lambda u, & \text{if }x\in\Gamma_{\alpha}\backslash\mV,\alpha\in A,\\[4pt]
{\displaystyle \sum_{\alpha\in\mA_{i}}\gamma_{i\alpha}\mu_{\alpha}\partial_{\alpha}v(\nu_i)\ge\sum_{\alpha\in\mA_{i}}\gamma_{i\alpha}\mu_{\alpha}\partial_{\alpha}u(\nu_i)}, & \text{if }\nu_{i}\in\mV,
\end{cases}
\end{equation*}
then $v\ge u$.
\end{lemma}

We now give a continuous dependence estimate for the solution of \eqref{eq:HJ} and \eqref{eq:HJ_discount} with respect to the data of the problem.
\begin{theorem}\label{thm:dip_continua}
For $i=1,2$, consider $H^i:\G\times\R\to\R$ and $F^i:\Gamma\rightarrow \R$ such that $H^{i}_\a(x,p)=\sup_{a\in A}\left\{ - b^{i}_\a(x,a)p-f^{i}_\a(x,a)\right\}$ for $x\in\Gamma_\a$, $\a\in \mA$.
Assume that
\begin{itemize}
\item[(i)]  $b^i_\a$, $f^i_\a$, $i=1,2$, satisfy (H2) with the same constants $K$, $L$;
\item[(ii)]  $\|H^i_\a\|_{C^{1,\t}(\G_\a\times (-\bar K,\bar K))}\le K_H$ for $\a \in\mA$, $i=1,2$, where $\bar K$ as in \eqref{eq:bound_2} and $\t\in (0,1]$;
\item[(iii)] for some $\theta\in(0,1]$, the functions $F^i:\Gamma\rightarrow \R$, $i=1,2$, fulfill
\begin{equation*}
\|F^i_\a\|_{L^\infty(\Gamma_\a)}\leq K_F\quad\textrm{and}\quad [F^i_\a]_{\theta,\Gamma_\a}\leq L_F \qquad \forall \a \in\mA.
\end{equation*}
\end{itemize}
For $i=1,2$, let $\vl_i$ be the solution of \eqref{eq:HJ_discount} with Hamiltonian $H(x,p)=H^i(x,p)+F^i(x)$ and set $\wl_i:=\vl_i-\med{{\vl_i}}$. Then, there exists a positive constant $C_0$, independent of $\l$, such that
\begin{equation} \label{eq:dc_0}
\begin{split}
\|\wl_1-\wl_2\|_{C^2(\G)}&\le  C_0\max_{\a\in\mA}\Big( \max_{x,a}|b^1_\a(x,a)-b^2_\a(x,a)|\\
&+\max_{x,a}|f^1_\a(x,a)-f^2_\a(x,a)|+\max_{x}|F^1_\a(x)-F^2_\a(x)|\big)+\\
&\max_{\a\in\mA}[H^1_\a-H^2_\a]_{1,\G_\a\times (-\bar K,\bar K)}+\max_{\a\in\mA}[F^1_\a-F^2_\a]_{\theta,\Gamma_\a}.
\end{split}
\end{equation}
Estimate \eqref{eq:dc_0} also holds for $v_i$, $i=1,2$,  solution to \eqref{eq:HJ} corresponding to $H(x,p)=H^i(x,p)+F^i(x)$.
\end{theorem}
\begin{proof} 
We shall proceed by contradiction. We assume that, for $k\to +\infty$, there exist sequences $\l_k\to 0$ , $b^{i,k}_\a$, $f^{i,k}_\a$, $F^{i,k}_\a$, $i=1,2$, satisfying (H2) and $(iii)$ with the same constants $K$, $L$, $\theta$, $K_F$ and $L_F$ and $v^{\l_k}_i$, $i=1,2$, solution to \eqref{eq:HJ_discount} with discount $\l_k$ and coefficients $b^{i,k}_\a$, $f^{i,k}_\a$ and $F^{i,k}_\a$ such that   
\begin{align*}
c_k:=&\|w^{\l_k}_1-w^{\l_k}_2\|_{C^2(\G)}\\
  \ge&  k\max_{\a\in\mA}\Big( \max_{x,a}|b^{1,k}_\a(x,a)-b^{2,k}_\a(x,a)|\\ 
&+\max_{x,a}|f^{1,k}_\a(x,a)-f^{2,k}_\a(x,a)|+\max_{x}|F^{1,k}_\a(x)-F^{2,k}_\a(x)|\Big)\\
&+\max_{\a\in\mA}[H^{1,k}_\a-H^{2,k}]_{1,\G_\a\times (-\bar K,\bar K)}+\max_{\a\in\mA}[F^{1,k}_\a-F^{2,k}_\a]_{\theta,\Gamma_\a}
\end{align*}
where  $w^{\l_k}_i=v^{\l_k}_i-\med{v^{\l_k}_i}$ and 
\[H^{i,k}_\a(x,p)=\sup_{a\in A}\left\{ - b^{i,k}_\a(x,a)p-f^{i,k}_\a(x,a)\right\},\qquad x\in\Gamma_\a,\, i=1,2.\]
The function $w^{\l_k}_i=v^{\l_k}_i-\med{v^{\l_k}_i}$, $i=1,2$, solves the equation
\begin{equation*} 
\begin{cases}
-\mu_{\alpha}\partial^{2}w^{\l_k}_i+H^{i,k}\big(x,\partial w^{\l_k}_i\big)+F^{i,k}_\a(x)+\l_k w^{\l_k}_i+\l_k\med{v^{\l_k}_i}=0, & x\in\left(\Gamma_{\alpha}\backslash\mV\right),\alpha\in\mA,\\[3pt]
{ \sum_{\alpha\in\mA_{i}}\gamma_{i\alpha}\mu_{\alpha}\partial_{\alpha} w^{\l_k}_i(\nu_i)=0,} 
& \nu_{i}\in\mV,\\[3pt]
w^{\l_k}_i|_{\Gamma_{\alpha}}(\nu_i)=w^{\l_k}_i|_{\Gamma_{\beta}}(\nu_i), & \alpha,\beta\in\mA_{i},\nu_{i}\in\mV.
\end{cases} 
\end{equation*}
Hence the function $W^k=c_k^{-1}(w^{\l_k}_1-w^{\l_k}_2)$ solves the equation
\begin{equation}\label{eq:diff_w1}
\left\{
\begin{array}{ll}
-\m_\a\partial^2 W^k+c_k^{-1}\big(H^{1,k}(x,\partial w^{\l_k}_{1} )\\
\qquad \qquad-H^{1,k}(x,\partial w^{\l_k}_{2} )\big)
+R^k=0, & x\in\left(\Gamma_{\alpha}\backslash\mV\right),\alpha\in\mA,\\[3pt]
{\sum_{\alpha\in\mA_{i}}\gamma_{i\alpha}\mu_{\alpha}\partial_{\alpha} W^k(\nu_i)=0,} 
& \nu_{i}\in\mV,\\[3pt]
W^K|_{\Gamma_{\alpha}}(\nu_i)=W^k|_{\Gamma_{\beta}}(\nu_i), & \alpha,\beta\in\mA_{i},\nu_{i}\in\mV,
\end{array}
\right.
\end{equation}
 where
\begin{multline*}
R^k=\l_kc_k^{-1}(\med{v^{\l_k}_1} -\med{v^{\l_k}_2})+c_k^{-1}\big(H^{1,k}(x,\partial w^{\l_k}_{2} )-H^{2,k}(x,\partial w^{\l_k}_{2} ))\\
\l_kW^k+c_k^{-1}\left(F^{1,k}(x)-F^{2,k}(x)\right).
\end{multline*}
Since $H^{i,k}_\a$ belongs to $C^1(\G_\a\times \R)$ for all $\a\in\mA$, we rewrite \eqref{eq:diff_w1} as
\begin{equation}\label{eq:diff_w2}
\left\{
\begin{array}{ll}
-\m_\a\partial^2 W^k+g^k\partial W^k+R^k=0, & x\in\left(\Gamma_{\alpha}\backslash\mV\right),\alpha\in\mA,\\[3pt]
{ \sum_{\alpha\in\mA_{i}}\gamma_{i\alpha}\mu_{\alpha}\partial_{\alpha} W^k(\nu_i)=0,} 
& \nu_{i}\in\mV,\\[3pt]
W^k|_{\Gamma_{\alpha}}(\nu_i)=W^k|_{\Gamma_{\beta}}(\nu_i), & \alpha,\beta\in\mA_{i},\nu_{i}\in\mV,
\end{array}
\right.
\end{equation}
where 
\[g^k_\a(x)=\int_0^1\partial_p H^{1,k}_\a\left(x,t \partial w^{\l_k}_{1,\a}+(1-t)\partial w^{\l_k}_{2,\a}\right)dt\]
and we aim to pass to the limit in \eqref{eq:diff_w2} for $k\to \infty$.\\
We first observe that, since $w^{\l_k}_i\in C^{2,\theta}(\G)$ with $\|w^{\l_k}_i\|_{C^{2,\theta}}\leq \bar K$ and $H^{i,k}_\a\in C^{1,\t}(\G_\a\times [-\bar K,\bar K])$, then the functions $g^k_\a$, $\a\in\mA$, are uniformly bounded and  H\"older continuous of exponent $\t$. Hence, there exists $g:\G\to\R$ such that for any $\a\in\mA$
\begin{equation}\label{eq:conv_g}
g^k_\a\to g_\a \qquad \text{for $k\to \infty$, uniformly in $\G_\a$}.
\end{equation} 
Moreover, we claim that the function $R^k$ is uniformly $\theta$-H\"older continuous. Indeed, by \eqref{eq:bound_2} and the definition of $c_k$, we have
\begin{align*}
&c_k^{-1}\left[H^{1,k}_\a(x,\partial w^{\l_k}_{2} )-H^{2,k}_\a(x,\partial w^{\l_k}_{2} )\right]_{1,\G_\a}\\
& \le\left(1+\|w^{\l_k}_2\|_{C^{2,\th}(\G)}\right)\, c_k^{-1} [H^1-H^2]_{1,\G_\a\times (-\bar K,\bar K)} 
\le 1+\|w^{\l_k}_2\|_{C^{2,\th}(\G)}\le 1+\bar K.
\end{align*}
On the other hand, by our choice of $c_k$, we have
\begin{equation*}
\left[c_k^{-1}\left(F^{1,k}_\a-F^{2,k}_\a\right)\right]_{\theta,\Gamma_\a}=c_k^{-1}\left[F^{1,k}_\a-F^{2,k}_\a\right]_{\theta,\Gamma_\a}\leq 1;
\end{equation*}
hence, our claim is proved.\\
We now claim that
\begin{equation}\label{eq:stima1_R}
\|R^k\|_{L^\infty(\G)}=o_k(1)\qquad \text{as $k\to\infty$}
\end{equation}
where $\lim_{k\to \infty}o_k(1)=0$, uniformly in $x$ and may change from line to line. 
Indeed, we have
\begin{equation}\label{stimaR_1}
\l_kW^k=\l_k\frac{w^{\l_k}_{1}-w^{\l_k}_{2}}{\|w^{\l_k}_{1}-w^{\l_k}_{2}\|_{C^2(\G)}}=o_k(1).
\end{equation}
Moreover, we have
\begin{multline}\label{eq:dc_1}
\l_k\|v^{\l_k}_1-v^{\l_k}_2\|_{L^\infty}\le  \max_{\a\in\mA} \big(\bar K \max_{x,a}|b^{k,1}_\a(x,a)-b^{k,2}_\a(x,a)|+\\\max_{x,a}|f^{k,1}_\a(x,a)-f^{k,2}_\a(x,a)|+\max_{x}|F^{1,k}_\a(x)-F^{2,k}_\a(x)|\big),
\end{multline}
where $\bar K$ as in \eqref{eq:bound_2}. Indeed, to prove \eqref{eq:dc_1}, it is sufficient to observe that 
\begin{align*}
v_{\pm}(x)&= v^{\l_k}_2(x)\pm \l_k^{-1}\max_{\a\in\mA}\big( \bar K\max_{x,a}|b^{k,1}_\a(x,a)-b^{k,2}_\a(x,a)|\\
&+\max_{x,a}|f^{k,1}_\a(x,a)-f^{k,2}_\a(x,a)|+\max_{x}|F^{1,k}_\a(x)-F^{2,k}_\a(x)|\big) 
\end{align*}
are a subsolution and a supersolution of the equation satisfied by $v^{\l_k}_1$ and to apply  Lemma~\ref{lem:comp_prin}. By \eqref{eq:dc_1} and (H2), we have
\begin{equation}\label{stimaR_2}
\begin{split} 
|\l_k c_k^{-1}(\med{v^{\l_k}_1} -\med{ v^{\l_k}_2})|\le  c_k^{-1}\max_{\a\in\mA} \big( \bar K\max_{x,a}|b^{1,k}_\a(x,a)-b^{2,k}_\a(x,a)|\\
+\max_{x,a}|f^{1,k}_\a(x,a)-f^{2,k}_\a (x,a)|+\max_{x}|F^{1,k}_\a(x)-F^{2,k}_\a(x)|\big)\int_\G dx= o_k(1).	
\end{split}
\end{equation}
Furthermore, taking into account  \eqref{eq:bound_2} and (H2), we have
\begin{align*} 
c_k^{-1}&\big(H^{1,k}(x,\partial w^{\l_k}_{2} )-H^{2,k}(x,\partial w^{\l_k}_{2} ))\\
&\le c_k^{-1}\max_{\a\in\mA}\big(\|\partial w^{\l_k}_{2}\|_{L^\infty(\G)} \max_{x,a}|b^{1,k}_\a(x,a)-b^{2,k}_\a(x,a)|\\
&+  \max_{x,a} |f^{1,k}_\a(x,a)-f^{2,k}_\a(x,a)| \big)\le
c_k^{-1}\max_{\a\in\mA}\big( \bar K\max_{x,a}|b^{1,k}_\a(x,a)-b^{2,k}_\a(x,a)|\\
&+ \max_{x,a} |f^{1,k}_\a(x,a)-f^{2,k}_\a(x,a)| \big)=o_k(1);
\end{align*}
by our choice of $c_k$, we also have
\begin{equation*}
\left|c_k^{-1}\left(F^{1,k}_\a(x)-F^{2,k}_\a(x)\right)\right|\leq c_k^{-1}\max_{x}\left|F^{1,k}(x)-F^{2,k}(x)\right|\leq 1/k.
\end{equation*}
By these estimates, \eqref{stimaR_1} and \eqref{stimaR_2}, we obtain the claim \eqref{eq:stima1_R}
and we conclude that for any $\a\in\mA$
\begin{equation}\label{eq:conv_R}
R^k_\a\to 0\qquad\text{for $k\to \infty$, uniformly in $\G_\a$}.
\end{equation}
For  $k\to \infty$, $W^k$ uniformly converges to a function $W\in C^2(\G)$ along with all its derivatives up to order $2$. Moreover, taking into account \eqref{eq:conv_g} and \eqref{eq:conv_R},  $W$ is a solution to
\[  
\left\{
\begin{array}{ll}
-\m_\a\partial^2 W+g\partial W=0, & x\in\left(\Gamma_{\alpha}\backslash\mV\right),\alpha\in\mA,\\[3pt]
{\sum_{\alpha\in\mA_{i}}\gamma_{i\alpha}\mu_{\alpha}\partial_{\alpha} W _i(\nu_i)=0,} 
& \nu_{i}\in\mV,\\[3pt]
W_i|_{\Gamma_{\alpha}}(\nu_i)=W_i|_{\Gamma_{\beta}}(\nu_i), & \alpha,\beta\in\mA_{i},\nu_{i}\in\mV,
\end{array}
\right.
\]
By Lemma \ref{lem:max_princ}, it follows that $W$ is constant and,   since $\med{W^k}=0$ for all $k$, then also $\med{W}=0$. It follows that $W\equiv 0$ which gives a contradiction to $\|W^k\|_{C^2(\G)}=1$ for all $k\in\N$.\\
The estimate  for the solutions of the ergodic problem \eqref{eq:HJ} follows immediately from Prop. \ref{lemma:bounds_discount_ergodic}.(ii) and since \eqref{eq:dc_0} is independent of $\l$.
\end{proof}
\begin{remark}\label{rmk:dc_Linfinito}
When assumption~$(ii)$ drops, it is possible to prove a $L^\infty$-continuous dependence estimate. More precisely, assuming $(i)$ and $(iii)$ of Theorem~\ref{thm:dip_continua}, there exists a constant $C_0$ such that
\begin{equation} \label{eq:dc_infinito}
\begin{split}
\|\wl_1-\wl_2\|_{L^\infty(\G)}&\le  C_0\max_{\a\in\mA}\Big( \max_{x,a}|b^1_\a(x,a)-b^2_\a(x,a)|\\
&+\max_{x,a}|f^1_\a(x,a)-f^2_\a(x,a)|+\max_{x}|F^1_\a(x)-F^2_\a(x)|\big).
\end{split}
\end{equation}
Estimate \eqref{eq:dc_infinito} also holds for $v_i$, $i=1,2$,  solution to \eqref{eq:HJ} corresponding to $H(x,p)=H^i(x,p)+F^i(x)$. \\
The proof is similar (and simpler) as the one of Theorem~\ref{thm:dip_continua} so we shall omit it and we refer the reader to \cite[Theorem 2.1]{marchi}.
\end{remark}
\section{Quasi-stationary Mean Field Games on networks}\label{sec:quasi_stat}
Quasi-stationary Mean Field Games, introduced in \cite{mouzouni} (see also \cite{cristiani_et_altri}), modelize the case when the agent cannot predict the evolution of the population in the future, as in the classical MFG theory, but, at each instant, it decides its behaviour only on the basis of the information available at the current time. This feature leads to systems given by an evolutive Fokker-Planck equation and a stationary HJ equation (which in fact depends on time through the cost). More precisely, at each  time $t\in [0,T]$, given the distribution of the population $m(t)$, the representative agent assumes that it will not change in the future and solves an optimal control problem with long-run average cost functional 
\[
\rho(t)=\inf_{a}\,\liminf_{T\to \infty}\frac{1}{T}\EE_{y,t}\left[\int_t^T \left(f(Y_s,a_s)+F[ m(t)](Y_s)\right)ds\right]
\]
where $Y_s$ is a fictitious dynamics on the network such $Y_t=x$ and $F$ is an additional cost term which depends on the distribution of the agents. If the corresponding HJ equation, see \eqref{eq:HJ}, admits a smooth solution $v(t)$,  then the optimal feedback law $a^\star_t(x)=-\partial_pH(x,\partial v(t))$ gives the vector field governing the evolution of the distribution of  the population at time $t$. 
This leads to study a class of quasi-stationary MFG systems
\begin{equation}
\begin{cases}
-\mu_{\alpha}\partial^{2}v+H (x,\partial v)+\r=F[m(t)](x), & \left(x,t\right)\in\left(\Gamma_{\alpha}\backslash\mV\right)\times(0,T),\alpha\in \cA,\\
\partial_{t}m-\mu_{\alpha}\partial^{2}m -\partial\left(m\partial_{p}H (x,\partial v)\right)=0, & \left(x,t\right)\in\left(\Gamma_{\alpha}\backslash\mV\right)\times(0,T),\alpha\in \cA,\\
 \sum_{\alpha\in\mA_{i}}\gamma_{i\alpha}\mu_{\alpha}\partial_{\alpha}v(\nu_{i},t)=0, & (\nu_{i},t)\in\mV\times(0,T),\\
\sum_{\alpha\in\mA_{i}}\mu_{\alpha}\partial_{\alpha}m(\nu_{i},t)+n_{i\alpha}\partial_{p}H_\alpha (\nu_{i},\partial v_\a(\nu_i,t))m|_{\G_\a}(\nu_{i},t)=0, & (\nu_{i},t)\in\mV\times(0,T),\\
v|_{\G_\a}(\nu_{i},t)=v|_{\G_\b}(\nu_{i},t),\ \dfrac{m|_{\G_\a}(\nu_{i},t)}{\gamma_{i\alpha}}=\dfrac{m|_{\G_\b}(\nu_{i},t)}{\gamma_{i\beta}}, & \alpha,\beta\in\mA_{i},(\nu_{i},t)\in\mV\times(0,T),\\
\med{v}=0,\ m\left(x,0\right)=m_{0}(x), & x\in\Gamma,
\end{cases}\label{eq:MFG}
\end{equation}
where $H$ as in \eqref{eq:hamiltonian},  $m_0\in \cM$ describes the initial distribution of the players and $F$ is the nonlocal coupling cost (see below for the precise assumptions). These systems loss the standard forward-backward structure of MFG. In order to establish the existence of a solution, it is crucial to have some regularity in time for the value function~$v$. In the classical approach for MFG, such a regularity follows from the parabolicity of HJ equation; here, it will be retrieved using the continuous dependence estimate of Section~\ref{sec:continuous_dep_est}.\\
We first recall some basic results concerning the Fokker-Planck equation  
\begin{equation}\label{eq:FP}
\begin{cases}
\partial_{t}m-\mu_{\alpha}\partial^{2}m-\partial\left(bm\right)=0, & \text{in }(\Gamma_{\alpha}\backslash\mV)\times (0,T),~\alpha\in\mA,\\
\sum_{\alpha\in\mA_{i}}\mu_{\alpha}\partial_{\alpha}m(\nu_{i},t)
+ n_{i\alpha}b_\a(\nu_{i},t)m|_{\Gamma_{\alpha}}(\nu_{i},t)=0, & t\in(0,T),~\nu_{i}\in\mV,\\
\dfrac{m|_{\Gamma_{\alpha}}(\nu_{i},t)}{\gamma_{i\alpha}}=\dfrac{m|_{\Gamma_{\beta}}(\nu_{i},t)}{\gamma_{i\beta}}, &  t\in(0,T),~\alpha,\beta\in\mA_i,~ \nu_{i}\in\mV\backslash\partial\Gamma,\\
m\left(x,0\right)=m_{0}(x), & x\in\Gamma.
\end{cases}
\end{equation}
The vertex conditions for the FP equation are obtained by duality with respect to the corresponding vertex conditions for the HJ equation and  express conservation of the flux and,  respectively, a rule for the distribution of the density. \\
We introduce suitable parabolic spaces for weak solution of the FP equation. We set $V=H^1(\G)$ and
\[
H^{1}_{b}(\Gamma):=\left\{ v:\Gamma\rightarrow\R \text{ s.t. }v_{\alpha}\in H^1 (0,\ell_{\alpha})\text{ for all }\alpha\in\mA\right\},
\]
(unlike $V$, continuity at the vertices is not required), endowed with the norm
\[
\left\Vert v\right\Vert _{H^1_{b}(\Gamma)}=\left( \sum_{\alpha\in\mA}
 \Vert \partial v_{\alpha} \Vert _{L^{2}\left(0,\ell_{\alpha}\right)}^{2}+
\left\Vert v\right\Vert _{L^2(\Gamma)}^{2}\right)^{\frac 1 2}.
\]
By Remark~\ref{rmk:extension}, for $v\in H^{1}_{b}(\Gamma)$, we still denote $v_\alpha$ the extension by continuity of~$v_\alpha$ on the whole interval~$[0,\ell_\alpha]$.\\
We also define
\begin{equation*}
W:=\left\{ w:\Gamma\rightarrow\mathbb{R}:\;w\in H_{b}^{1}\left(\Gamma\right)  \text{  and }\frac{w|_{\Gamma_{\alpha}}(\nu_{i})}{\gamma_{i\alpha}}=\frac{w|_{\Gamma_{\beta}}(\nu_{i})}{\gamma_{i\beta}}\text{ for all }
i\in I, \; \alpha,\beta\in\mA_{i}\right\},
\end{equation*}
\begin{align*}
PC(\Gamma\times [0,T]):=\{&v:\Gamma\times [0,T]\to \R:\,v(\cdot,t)\in PC(\Gamma)\, \text{for all $t\in [0,T]$ and}\\
&\text{$v|_{\Gamma_\alpha\times [0,T]}\in C(\Gamma_\alpha\times [0,T])$ for all $\alpha\in \cA$}
\}.	
\end{align*}
\begin{definition}\label{def:sol_FP}For $m_0\in L^{2}(\G)$, a weak solution of \eqref{eq:FP} is a function $ m\in L^{2}\left(0,T;W\right)\cap C([0,T]; L^2(\Gamma))$ such that $\partial_{t}m\in L^{2}\left(0,T;V'\right)$ and
\begin{equation*}
\begin{cases}
\ds \langle \partial_{t}m,v\rangle _{V', V}+\int_{\Gamma}\mu\,\partial m\partial vdx+\int_{\Gamma}bm\partial vdx=0\quad\text{for all }v\in H^1(\G),\text{ a.e. }t\in (0,T),\\
m\left(\cdot,0\right)=m_{0}.
\end{cases}
\end{equation*}
\end{definition}
The following result concerns existence, uniqueness and stability for the solution of \eqref{eq:FP}
(see \cite[Theorem 3.1 and Lemma 3.1]{adlt2})
\begin{proposition}\label{thm: existence and uniqueness FK}
We have
\begin{itemize}
\item[(i)]  For $ b\in L^\infty (\Gamma \times (0,T))$, $m_0\in L^2(\Gamma)$,  there exists a unique 
weak solution to \eqref{eq:FP}. Moreover, there exists  $C=C(\Vert b\Vert _{\infty})$ such that 
\begin{equation}\label{eq:stima_FP}
\left\Vert m\right\Vert _{L^{2}\left(0,T;W\right)}  
+\left\Vert m\right\Vert _{L^{\infty}\left(0,T; L^2\left(\Gamma\right)\right)}+
\left\Vert \partial_{t}m\right\Vert _{L^{2}\left(0,T;V'\right)}\le C\left\Vert m_{0}\right\Vert _{L^2(\Gamma)}.
\end{equation}
Moreover, if $m_0\in \cM$, then $m(t)\in \cM$ for all $t\in [0,T]$.
\item[(ii)] Let $b^n$ be  such that
\begin{align*}
 b^n\longrightarrow b\text{ in }
L^{2}\left(\Gamma \times (0,T)\right),\quad\|b\|_{L^{\infty}\left(\Gamma \times (0,T)\right)}, \|b^n\|_{L^{\infty}\left(\Gamma \times (0,T)\right)}\le K 
\end{align*}
with $K$ independent of $n$.
Let $ m^n $  (respectively  $m$) be the solution of   \eqref{eq:FP} corresponding to  the coefficient $b^n$ (resp. $b$).
Then, the sequence  $(m^n)$ converges to $m$ in $L^{2}\left(0,T;W\right)\cap L^\infty\left(0,T;L^2(\Gamma)\right)$, and the sequence $(\partial_t m^n)$ converges to $(\partial_t m) $ in $L^{2}\left(0,T;V'\right)$. 
\end{itemize}
\end{proposition}
\begin{proposition}\label{prop:stima_wass}
For $m_0\in L^{2}(\G)\cap\cM$, let $m$ be the solution of \eqref{eq:FP} found in Proposition~\ref{thm: existence and uniqueness FK}. Then there exists a constant $C_W$, depending only on 
$\|b\|_{L^{\infty}}$ and $\|m_0\|_{L^2}$, such that
\begin{equation}\label{eq:stima_wass}
\Wass_1(m(t),m(s))\le C_W|t-s|^\half	
\end{equation}
\end{proposition}
\begin{proof}
Let $\phi:\G\to\R$ with $|\phi(x)-\phi(y)|\le d_\G(x,y)$, hence $\phi\in H^1(\G)$. For $s,t\in [0,T]$ with $s<t$, by Definition \ref{def:sol_FP} and regularity of $m$, we have
\begin{align*}
& \int_{\Gamma} \phi(x)(m(t)-m(s))dx \le\int_s^t\int_{\Gamma}(\mu|\partial m|\,|\partial \phi| +m|b|\,|\partial \phi| )dxdr\\	
&\le \|\mu\|_\infty\int_s^t\int_{\Gamma}|\partial m|dxdr+ \|b\|_{L^\infty} \int_s^t\int_{\Gamma}m\,dxdr\\
&\le \|\mu\|_\infty \left[\int_s^t\int_{\Gamma}|\partial m|^2dxdr\right]^\half \left[\int_s^t\int_{\Gamma}1dxdr\right]^\half+
\|b\|_{L^\infty} \int_s^t\int_{\Gamma}mdxdr.
\end{align*}
Exploiting $\int_{\Gamma}m(r)dx=1$ for any $r\in [0,T]$,  \eqref{eq:stima_FP} and \eqref{eq:wass_dist}, by the previous inequality we get \eqref{eq:stima_wass}.
\end{proof}

We now prove the well posedness of system \eqref{eq:MFG}.
\begin{theorem}
Assume (H1), (H2), $m_0\in L^2(\G)\cap\cM$, $H_\a\in C^{1,\tau}(\Gamma_\a\times(-\bar K,\bar K))$ (where $\bar K$ as in~\eqref{eq:bound_2}) and $F:\cM\to L^2(\G)$ satisfies
\begin{itemize}
\item[(F)]
$F_\a:\mM \to C^{0,\theta}(0,\ell_\a)$, $\a\in\mA$, and there exist $C_F>0$ and $\theta\in(0,1]$ s.t.
\[\begin{array}{l}
\max_{\a}\|F_\a[m]\|_{C^{0,\theta}}\le C_F,\\[4pt]
\max_{\a}\|F_\a[m_1]-F_\a[m_2]\|_{C^{0,\theta}}\le C_F\Wass_1(m_1,m_2)
\end{array} 
\]
for all $m$, $m_1$, $m_2$, $\a\in\mA$.  
\end{itemize} 
Then, the system \eqref{eq:MFG} admits a unique solution $(u,\r,m)$, where
$(u,\r)\in   C([0,T], C^2(\G))\times C([0,T])$ is a classical solution to the  HJ equation for any $t\in [0,T]$ and   $ m\in L^{2}\left(0,T;W\right)\cap C([0,T]; L^2(\Gamma)\cap\cM)$ with $\partial_{t}m\in L^{2}\left(0,T;V'\right)$ is a weak solution to the FP equation.
\end{theorem}
\begin{proof}\hfill\\
\textit{Existence: }
We consider the convex, compact set 
\[\mX=\Big\{m\in  C([0,T]; \cM):\,
\Wass_1(m(t),m(s))\le C_W|t-s|^\half, s,t \in [0,T]\Big\},\]
where $C_W$ as in \eqref{eq:stima_wass},  and we define a map $\mT: \mX\to\mX$ in the following way: given $m\in \mX$,  let $(u(t),\rho(t))$,  $t\in [0,T]$, be  the solution  of the HJ equation  
\begin{equation}\label{eq:MFG_HJ}
\begin{cases}
-\mu_{\alpha}\partial^{2}u+H (x,\partial u)+\r=F[m(t)](x),& 
x\in\left(\Gamma_{\alpha}\backslash\mV\right),\alpha\in \cA,\\
\sum_{\alpha\in\mA_{i}}\gamma_{i\alpha}\mu_{\alpha}\partial_{\alpha}u(\nu_{i})=0, & \nu_{i}\in\mV ,\\
u|_{\G_\a}(\nu_{i})=u|_{\G_\b}(\nu_{i}),&  \alpha,\beta\in\mA_{i},\nu_{i}\in\mV,\\
\med{u}=0 & x\in\Gamma.
\end{cases}
\end{equation}
Then, $\bm=\mT(m)$ solves the FP equation \eqref{eq:FP} with $b=\partial_p H(x,\partial u)$. \\ 
Note first that the map $\mT$ is well defined. Indeed,  thanks to  Prop.\ref{lemma:bounds_discount_ergodic}, Theorem \ref{thm:dip_continua} and (F), there exists $(u,\r)\in   C([0,T], C^2(\G))\times C([0,T])$
 which solves \eqref{eq:MFG_HJ} for any $t\in [0,T]$. Moreover, by  Prop. \ref{thm: existence and uniqueness FK}, there exists a unique solution $\bm$, in the sense of Def. \ref{def:sol_FP}, to  problem \eqref{eq:FP} with $b=\partial_p H(\partial u)$. Since $\bm\in C([0,T]; L^2(\Gamma))$ can be identified with the corresponding Borel measure with density $\bm(t)$ on $\G$ at time $t$, by Prop. \ref{prop:stima_wass}, we also have that
 $\mT$ maps $\mX$ into itself.\\
We prove that  $\mT$ is continuous. Given  $m_n$, $m\in \mX$, let $(u^n(t),\rho^n(t))$, $(u(t),\r(t))$ be the solutions, for any $t\in [0,T]$, of the HJ equations  \eqref{eq:MFG_HJ} with right hand side $F[m^n(t)]$ and, respectively, $F[m(t)]$ and let $\bm^n=\mT(m^n)$, $\bm=\mT(m)$. If $m^n\to m$ in $\mX$, then $\Wass_1(m^n(t),m(t))\to 0$ uniformly for $t\in (0,T)$. Invoking again Theorem~\ref{thm:dip_continua}, by \eqref{eq:dc_0} and (F), for any $t\in [0,T]$ there holds
\begin{align*}
\|u^n(t)-u(t)\|_{C^2(\G)}&\le C_0 \max_{\a\in\mA}\|F_\a[m^n(t)]-F_\a[m(t)]\|_{C^{0,\theta}}\\
&\le C \max_{ \a\in\mA} \Wass_1 (m^n_{\a}(t),m_\a(t)),
\end{align*}
with $C$ independent of $m^n$, $m$. The previous estimate and Prop. \ref{thm: existence and uniqueness FK}.(ii) with $b^n=\partial_p H(x,\partial u^n)$, $b=\partial_p H(x,\partial u)$ imply that $\bm^n$ converges to $\bm$ in $\mX$ and therefore the map $\mT$ is continuous.\\
By Schauder fixed point theorem, we conclude that there exists a fixed point of $\mT$ and therefore a solution of \eqref{eq:MFG}.

\medskip 
\par\noindent
\textit{Uniqueness: }
Suppose that there are two solutions $(u_1,\r_1,m_1)$, $(u_2,\r_2,m_2)$ of \eqref{eq:MFG}.\\
As in~\cite{adlt2}, we introduce the function $\varphi:\Gamma\rightarrow \R$ as
\begin{equation*}
\varphi_\alpha \textrm{ is affine on }[0,\ell_\alpha],\qquad \varphi_\alpha(\nu_i)=\gamma_{i\alpha} \quad \textrm{if }\alpha\in \mA_{i}. 
\end{equation*}
Note that $\varphi\in H^1_b(\Gamma)$ is strictly positive and bounded. Hence the reciprocal $\varphi^{-1}$ is well defined, positive and bounded; this property will play a crucial role in our argument.\\
Set $M=\varphi^{-1}(m_1-m_2)$. The transition condition of~$m_i$ and the definition of~$\varphi$ ensure that $M(t)\in H^1(\Gamma)$ for a.e. $t\in (0,T)$. Hence, we can use Definition~\ref{def:sol_FP} for $m_1$ and $m_2$ with $M(t)$ as a test function obtaining
\begin{multline}\label{eq:unic1}
\half\frac{d}{dt}\|(m_1-m_2)(t)(\varphi^{-1})^{1/2}\|^2_{L^2(\Gamma)}+\|\partial(m_1-m_2)(t)(\mu\varphi^{-1})^{1/2}\|^2_{L^2(\Gamma)}\\=-\int_\Gamma\mu\partial(m_1-m_2)(t)(m_1-m_2)(t)\partial(\varphi^{-1})dx\\-\int_\Gamma b_1(m_1-m_2)(t)\partial M(t) dx-\int_\Gamma (b_1-b_2)m_2(t)\partial M(t) dx 
\end{multline} 
where $b_i=\partial_p H(\cdot,\partial u_i(\cdot,t))$ for $i=1,2$. We now estimate the three integrals in the right hand side of equality~\eqref{eq:unic1}.
Since now on, $C$ will denote a constant that may change from line to line but is always independent of~$M$. By Cauchy-Schwarz inequality, we get
\begin{equation}\label{eq:unic1bis}
\begin{split}
&-\int_\Gamma\mu\partial(m_1-m_2)(t)(m_1-m_2)(t)\partial(\varphi^{-1})dx\\
&\leq \int_\Gamma\left|\mu\partial(m_1-m_2)(t)(m_1-m_2)(t)\partial(\varphi^{-1})\varphi^{1/2}\varphi^{-1/2}\right|dx\\
&
\leq \frac12 \|\partial(m_1-m_2)(t)(\mu\varphi^{-1})^{1/2}\|^2_{L^2(\Gamma)}
+\frac12\int_\Gamma \mu|m_1-m_2|^2(t)|\partial(\varphi^{-1})|^2\varphi dx\\
&\leq 
\frac12 \|\partial(m_1-m_2)(t)(\mu\varphi^{-1})^{1/2}\|^2_{L^2(\Gamma)}
+C\|(m_1-m_2)(t)(\varphi^{-1})^{1/2}\|^2_{L^2(\Gamma)}
\end{split}
\end{equation}
where the last inequality is due to the boundedness of $\mu$ and to the properties of~$\varphi$. Moreover, by the boundedness of $b_1$ and of~$\mu$, again using Cauchy-Schwarz inequality, we have
\begin{equation}\label{eq:unic1ter}
\begin{split}
&-\int_\Gamma b_1(m_1-m_2)(t)\partial M(t) dx\\
&=-\int_\Gamma b_1(m_1-m_2)(t)[\partial(m_1-m_2)(t)\varphi^{-1}-(m_1-m_2)(t)\partial(\varphi^{-1})]dx\\
&\leq\int_\Gamma \left|b_1(m_1-m_2)(t)\partial(m_1-m_2)(t)\varphi^{-1}\right|dx+\int_\Gamma \left|b_1(m_1-m_2)^2(t)\partial(\varphi^{-1})\right|dx\\
&\leq \frac14 \|\partial(m_1-m_2)(t)(\mu\varphi^{-1})^{1/2}\|^2_{L^2(\Gamma)}
+C\|(m_1-m_2)(t)(\varphi^{-1})^{1/2}\|^2_{L^2(\Gamma)}.
\end{split}
\end{equation}
Let us also assume for the moment the following estimate
\begin{multline}\label{eq:unic1quat}
-\int_\Gamma (b_1-b_2)m_2(t)\partial M(t) dx\leq \frac14 \|\partial(m_1-m_2)(t)(\mu\varphi^{-1})^{1/2}\|^2_{L^2(\Gamma)}
\\+C\|(m_1-m_2)(t)(\varphi^{-1})^{1/2}\|^2_{L^2(\Gamma)}
\end{multline}
whose proof is postponed at the end.

Replacing relations~\eqref{eq:unic1bis}, \eqref{eq:unic1ter}, \eqref{eq:unic1quat} in~\eqref{eq:unic1}, we get
\begin{equation*}
\frac{d}{dt}\|(m_1-m_2)(t)(\varphi^{-1})^{1/2}\|^2_{L^2(\Gamma)}\leq C\|(m_1-m_2)(t)(\varphi^{-1})^{1/2}\|^2_{L^2(\Gamma)} .
\end{equation*}
Since $m_1(0)=m_2(0)$, by the previous inequality we get $m_1(t)=m_2(t)$ for all $t\in [0,T]$, hence
  $u_1=u_2$ in $\G\times [0,T]$ and $\r_1=\r_2$.

It remains only to prove inequality~\eqref{eq:unic1quat}. To this end, we first estimate
\begin{eqnarray*}
\|b_1-b_2\|_{L^\infty(\Gamma)}&=&
\|\partial_p H(\cdot,\partial u_1(\cdot,t))-\partial_p H(\cdot,\partial u_2(\cdot,t))\|_{L^\infty(\Gamma)}\\
&\leq& C \|\partial u_1(\cdot,t)-\partial u_2(\cdot,t)\|_{L^\infty(\Gamma)}.
\end{eqnarray*}
Moreover, applying Theorem~\ref{thm:dip_continua} on $H_i(x,p)=H(x,p)-F[m_i(t)]$,we get 
\begin{equation*}
\|\partial u_1(\cdot,t)-\partial u_2(\cdot,t)\|_{L^\infty(\Gamma)}\leq
C\Wass_1 (m_1(t),m_2(t)).
\end{equation*}
By the last two inequalities, for $\d=\Wass_1 (m_1(t),m_2(t))$, we get $\|b_1-b_2\|_{L^\infty(\Gamma)}\leq C \d$ and we deduce
\begin{multline}\label{eq:unic2}
\int_\Gamma \left|(b_1-b_2)m_2(t)\partial M(t) \right|dx\leq C \int_\Gamma \d \left|m_2(t)\partial M(t) \right|dx\\ 
\qquad\leq \frac18 \int_\Gamma \mu|\partial M(t)|^2\varphi dx
+C \int_\Gamma \frac{\d^2 |m_2(t)|^2}{\mu\varphi}dx. 
\end{multline}
We denote $I_1$ and $I_2$ respectively the two integrals in the right hand side of the last inequality. We have
\begin{eqnarray*}
I_1&\leq& 2 \int_\Gamma \left[\mu|\partial (m_1-m_2)|^2\varphi^{-1} +\mu|m_1-m_2|^2|\partial (\varphi^{-1})|^2\varphi\right]dx\\
&\leq& 2 \|\partial(m_1-m_2)(\mu\varphi^{-1})^{1/2}\|^2_{L^2(\Gamma)}+
C\|(m_1-m_2)(\varphi^{-1})^{1/2}\|^2_{L^2(\Gamma)}.
\end{eqnarray*}
Moreover, since $m_2\in C([0,T],L^2(\Gamma))$, we have
\begin{equation*}
I_2=C\d^2\int_\Gamma |m_2|^2dx\leq C\d^2\leq C\|m_1-m_2\|^2_{L^2(\Gamma)}\leq C \|(m_1-m_2)(\varphi^{-1})^{1/2}\|^2_{L^2(\Gamma)}
\end{equation*}
where we used the definition of~$\d$ and the properties of~$\varphi$. Replacing these estimates for~$I_1$ and~$I_2$ in~\eqref{eq:unic2}, we accomplish the proof of inequality~\eqref{eq:unic1quat}.
\end{proof}

\section{Homogenization of HJ equations defined on a  lattice structure}\label{sec:homog}
In this section, we describe an application of the continuous dependence estimate in Section \ref{sec:continuous_dep_est} to the study of a homogenization problem for a HJ equation defined on a periodic network.\\
For $\e\in (0,1]$, let  $\G^\e$ be the periodic network generated by the lattice $\e \ZZ^N$. Hence
$\cV^\e=\e \ZZ^N$ and $\cE^\e =\left\{ \Gamma^\e_\a, \a\in\mA^\e\right\}$, where
\[
\Gamma^\e_\a=\left\{  ym+ (\e-y)n:\, y\in (0,\e)\right\}
\] for some $m,n\in\Z^N$ with  $|m-n|=1$.  Since  $\G$ is a lattice, there are $2N$ edges   coming out of each vertex $\nu_i\in \mV^\e$, in the directions of the vectors $e_k$ of the canonical basis of $\R^N$ and in the opposite directions $e_{-k}$.\\
For $k\in\ZZ^N$, we define
$\Gamma^\e_\a+k=\{ y(m+k)+(\e-y)(n+k):\, y\in (0,\e) \}$ and
we say that a function $\phi:\G^1\to\R$ is $\G^1$-periodic   if
\[
\phi_\b=\phi_\a \quad \text{if $\G^1_\b=\G^1_\a+k$, $k\in\ZZ^N$.}
\]
On the network $\G^\e$, we consider the problem
\begin{equation}\label{eq:HJ_hom}
\begin{cases}
-\mu_{\alpha}\partial^{2}\ue+H\left(x,\frac{x}{\e},\partial \ue\right)+\ue=0, & x\in\left(\Gamma^\e_\alpha\backslash\mV^\e\right),\alpha\in\mA^\e,\\
{\displaystyle \sum_{\alpha\in\mA_{i}}\gamma_{i\alpha}\mu_{\alpha}\partial_{\alpha}\ue(\nu_i)=0,} 
& \nu_{i}\in\mV^\e,\\
\ue|_{\Gamma_{\alpha}}(\nu_i)=\ue|_{\Gamma_{\beta}}(\nu_i), & \alpha,\beta\in\mA^\e_{i},\nu_{i}\in\mV^\e.
\end{cases}
\end{equation}
with 
\[
H_\a(x,y, p)=\sup_{a\in A}\left\{ - b_\a(x,y,a)p-f_\a(x,y,a)\right\}.
\]
We assume that $A$ is a compact metric space and $b,f :\R^N\times\G^1\times\R\to \R$ satisfy
\begin{itemize}
\item[(i)] $b_\a,f_\a:\R^N\times\G^1_\a\times\R\to \R$, $\a\in\mA$, are continuous and
there exist $K,L>0$ such that for $a\in A$, $x_1, x_2\in \R^N$ and $y_1,y_2\in\G^1_\a$, $\a\in\mA$,  there holds for $\phi_\a=b_\a,f_\a$ 	
\begin{align*}
&|\phi_\a(x_1,y_1,a)|\leq K,\\
& |\phi_\a(x_1,y_1,a)-\phi_\a(x_2,y_2,a)|\leq L(|x_1-x_2|+|y_1-y_2|); 
\end{align*}
\item[(ii)]  $b(x,\cdot,p)$, $f(x,\cdot,p)$ are $\G^1$-periodic;
\item[(iii)]   $\m_\a$ and $\g_{i,\a}$ only depend on the direction~$e_k$  parallel to~$\G_\a$ and $\g_{i,\a}=\g_\a$, for $\a\in \mA$ and $i\in I$.
\end{itemize}
We denote with $\mS^N$ the space of the symmetric $N\times N$ matrices.\\
We consider the effective problem   
\begin{equation}\label{eq:HJ_eff}
u +\bar H(x,Du, D^2u)=0 \qquad x\in\R^N,
\end{equation}
where the effective Hamiltonian $\bar H$ is defined as follows: for every $(x,P,X)\in\R^N\times\R^N\times \mS^N$ fixed, the value $\bar H(x,P,X)$ is  equal to $-\rho$,  where $\rho$   is the unique constant for which there exists a couple $(v,\r)$, with $v$ $\G^1$-periodic and $\r\in\R$, solution to
\begin{equation}\label{eq:HJ_cell}
\begin{cases}
-\mu_{\alpha}\partial^{2}(v+Xy\cdot y/2)+H(x,y,\pd( P\cdot y))+\rho=0, &  y\in\left(\Gamma^1_{\alpha}\backslash\mV^1\right),\alpha\in\mA^1,\\
{\displaystyle \sum_{\alpha\in\mA_{i}}\gamma_{i\alpha}\mu_{\alpha}\partial_{\alpha}v(\nu_i)=0,} 
& \nu_{i}\in\mV^1,\\
v|_{\Gamma_{\alpha}}(\nu_i)=v|_{\Gamma_{\beta}}(\nu_i), & \alpha,\beta\in\mA^1_{i},\nu_{i}\in\mV^1.
\end{cases}
\end{equation}
To display the dependence of $v$ with respect to $(x,P,X)$, we will denote with $v(\cdot;x,P,X)$ the  solution to \eqref{eq:HJ_cell}.\\

We need some preliminary results whose proof are postponed to the Appendix; note that the bound~\eqref{eq:corr_prp2} relies on the continuous dependence estimates of Sect.~\ref{sec:continuous_dep_est}.

\begin{lemma}\label{lem:Ham_eff}
For any $(x,P,X)\in\R^N\times\R^N\times \mS^N$, there is a unique $\r\in\R$ for which there exists a $\G^1$-periodic solution to \eqref{eq:HJ_cell}. Moreover 
\begin{equation}\label{eq:H_eff}
\rho=-\frac{\sum_{k=1}^N\g_k\Big[-Xe_k\cdot e_k+\int_{e_k}H(x,y,P\cdot e_k) dy\Big]}{\sum_{k=1}^N\g_k}
\end{equation}
and there exists a constant $\bar C_1$ such that
\begin{eqnarray}\label{eq:corr_prp1}
&&\|v(\cdot;x_1, P_1,X_1)\|_{C^{2,\th}(\G)}\le \bar C_1(1+|P_1|+|X_1|)\\\notag 
&&\|v(\cdot;x_1,P_1,X_1)-v(\cdot;x_2,P_2,X_2)\|_{L^\infty(\G)}\le \bar C_2(|P_1-P_2|+ |X_1-X_2|)\\ \label{eq:corr_prp2}
&&+ \bar C_1|x_1-x_2|(1+|P_1|\wedge|P_2|+|X_1|\wedge|X_2|)
\end{eqnarray}
for every $(x_1,P_1,X_1)$, $(x_2,P_2,X_2)\in\R^N\times\R^N\times \mS^N$, where $\th\in (0,1]$ as in \eqref{eq:bound_3}.
\end{lemma}
\begin{remark}
The formula~\eqref{eq:H_eff} entails that the effective operator $\bar H$ is convex and uniformly elliptic in $X$ and there exists a constant ${\bar C}_1$ such that
\begin{align*}
\left|\bar H(x_1,P_1,X_1)-\bar H(x_2,P_2,X_2)\right|\le {\bar C}_1(|P_1-P_2|+|X_1-X_2|)\\
+C|x_1-x_2|(1+|P_1|\wedge|P_2|+|X_1|\wedge|X_2|)
\end{align*}	
for every $(x_1,P_1,X_1)$, $(x_2,P_2,X_2)\in\R^N\times\R^N\times \mS^N$.
\end{remark}
\begin{lemma}\label{lem:HJ_eff}
The problems \eqref{eq:HJ_hom} and \eqref{eq:HJ_eff} admit a unique bounded solution $\ue$ and, respectively,   $u$. Moreover, there exists a constant $C_0$ such that
\begin{align}
\|\ue\|_{C^{2,\th}(\G^\e)}\le C_0 , \qquad
\|u\|_{C^{2,\d}(\R^N)}\le C_0 \label{eq:est_sol_HJ_eff}
\end{align}
for some  $\th,\,\d\in (0,1]$. 
\end{lemma}
\begin{lemma}\label{lem:kirchhoff}
If $g:\R^N\to\R$ is a smooth function, then it satisfies the Kirchhoff condition at  $\nu_i\in\mV^\e$ in~\eqref{eq:HJ_hom}, i.e.
\begin{equation*}
\sum_{\alpha\in\mA_{i}}\gamma_{i\alpha}\mu_{\alpha}\partial_\a g(\nu_i)=0.
\end{equation*}
\end{lemma}
\begin{theorem}\label{thm:est_rate}
Let $\ue$ and $u$ be respectively  the solution of \eqref{eq:HJ_hom} and \eqref{eq:HJ_eff}. Then, there exists a constant $M$ such that for $\e$ sufficiently small
\begin{equation*}
\|u^\e-u\|_{L^\infty(\G^\e)}\le M\e^\d,
\end{equation*}
where $\d $ as in \eqref{eq:est_sol_HJ_eff}.
\end{theorem}

\begin{proof}
Given $\e\in (0,1)$, for $\eta\in(0,\infty)$ define the function 
\begin{equation*} 
\phi(x):=\ue(x)-u(x)-\e^2v\left(\frac x\e; [u](x)\right)-\frac\eta 2 |x|^2\qquad\forall x\in \G^\e,
\end{equation*}
where $v\left(y; [u](x)\right):=v(y;x,Du(x),D^2u(x))$ is the solution of \eqref{eq:HJ_cell} with $(x,P,X)=(x,Du(x), D^2u(x))$. Since  $u$, $\ue$ and $v$ are bounded, there exists $\hx\in \G^\e$ where the function $\phi$ attains its maximum. \\
Set $c:=4\bar C_1(1+2C_0)\e^\d$ and introduce the function
\begin{equation*} 
\tphi(x):=\ue(x)-u(x)-\e^2v\left(\frac x\e; [u](\hx)\right)-\frac\eta 2 |x|^2-c|x-\hx|^2 
\end{equation*}
where $|x-\hx|$ is the standard Euclidean distance between $x$ and $\hx$.
We have $\tphi(\hx)=\phi(\hx)$ and also, by the definition of $\hx$,
\begin{eqnarray*}\tphi(\hx)-\tphi(x)&=&[\tphi(\hx)-\phi(x)]+[\phi(x)-\tphi(x)]\geq
\phi(x)-\tilde\phi(x)\\
&\geq&-\e^2\left[v(\frac x\e;[u](x))-v(\frac x\e;[u](\hx))\right]+c\e^2
\end{eqnarray*}
for every $x\in \partial B(\hx,\e)\cap\G_\e$, where $B(\hx,\e)=\{x\in\R^N:|x-\hx|<\e\}$. 
Using the estimates in~\eqref{eq:corr_prp2}, Lemma~\ref{lem:HJ_eff} and recalling the definition of $c$, we get
\begin{align*}
\tilde\phi(\hx)-\tilde\phi(x) \geq & -\bar C_1[2C_0\e^\d 
+(1+\|Du\|_\infty+\|D^2u\|_\infty)\e]\e^2\\
&+4\bar C_1(1+2C_0)\e^{2+\d}> 0
\end{align*}
for every $x\in \partial B(\hx,\e)\cap\G^\e$.
Therefore, $\tilde \phi$ attains a maximum at some point $\tx\in B(\hx,\e)\cap\G^\e$. \\
Let us prove that there exists a constant $M_1>0$ (independent of $\e$ and $\eta$) such that
\begin{equation}\label{eq:stima_hx}
\eta^{\frac 12} |\tx|  \leq M_1.
\end{equation}
By Lemma~\ref{lem:Ham_eff}, Lemma~\ref{lem:HJ_eff}  and the inequality $\phi(\hx)\geq \phi(0)$, we obtain
\begin{equation*}
\frac \eta 2 |\hx|^2 \leq 4C_0 +2\bar C_1(1+2C_0)\e^2.
\end{equation*}
We deduce that, for $M_1$ sufficiently large, we have $\eta^{1/2}|\hx| \leq M_1/2$ and therefore 
$$
\eta^{\frac 12}|\tx|\leq \eta^{\frac 12} |\hx|+\eta^{\frac 12} |\tx-\hx|\leq \eta^{\frac 1 2}\frac{M_1}{2}+\eta^{\frac 12}\e\leq \eta^{\frac 1 2}M_1,
$$
hence  \eqref{eq:stima_hx}.\\
We claim  that there exists a constant~$M$ (independent of $\e$ and $\eta$) such that
\begin{equation}\label{eq:main_est}
u^\e(\tx)-u(\hx)\leq M\left[\e^\d +\eta^{1/2}\right].
\end{equation}
We first show that $\tx\not\in\mV^\e$. Indeed, assume by contradiction that $\tx=\nu_i\in\mV^\e$. By adding the term $-d_\G(x,\tx)^2$, where $d_\G$ is the geodesic distance on the network, it is not restrictive to assume that $\tx$ is a strict maximum point for $\tilde\phi$ and therefore $\partial_\a \tilde\phi(\nu_i)>0$ for all $\a \in\mA^\e_i$ (recall the definition of $\partial_\alpha$ as the outward derivative at the vertex). Since $u^\e$ and $v$ solve respectively \eqref{eq:HJ_hom} and \eqref{eq:HJ_eff}, by Lemma \ref{lem:kirchhoff} we have
\begin{equation*} 
0< \sum_{\alpha\in\mA_{i}}\gamma_{i\alpha}\mu_{\alpha}\partial_{\alpha}\phi(\nu_i)=
\sum_{\alpha\in\mA_{i}}\gamma_{i\alpha}\mu_{\alpha}\partial_\a\left( u+\frac\eta 2 |x|^2+c|x-\hx|^2\right)_{x=\nu_i}=0,
\end{equation*}
a contradiction and therefore $\tx\in (B(\hx,\e)\cap\G^\e)\setminus \mV^\e$. Let $\a\in\mA$ be such that $\tx\in\G^\e_\a$ and $e_\a$ a unit vector parallel to $\G^\e_\a$. Since $\ue$ satisfies  \eqref{eq:HJ_hom}   and $\tx$ is a maximum point for $\ue(x)-[u(x)+\e^2 v(x/\e;[u](\hx))+\eta|x|^2/2+c|x-\hx|^2]$, we have
\begin{equation}\label{eq:test_equation}
\begin{split}
0\geq &u^\e(\tx)- \mu_\a \pd^2\left[ u(x)+\e^2 v(\frac{x}{\e};[u](\hx))+\frac\eta 2|x|^2+c|x-\hx|^2\right]_{x=\tx}+\\
&H\Big(\tx,\frac{\tx}{\e}, \pd\left[u(x)+\e^2   v(\frac{x}{\e};[u](\hx))+\frac\eta 2|x|^2+c|x-\hx|^2\right]_{x=\tx}\Big).
\end{split}
\end{equation}
We compute
\begin{align*}
&\pd\big[u(x)+\e^2   v(\frac{x}{\e};[u](\hx))+\frac\eta 2|x|^2+c|x-\hx|^2\big]_{x=\tx} \\
=& Du(\tx)\cdot e_\a+ \e \pd_y  v(\frac{\tx}{\e};[u](\hx))+\eta\tx\cdot e_\a +2c(\tx-\hx)\cdot e_\a,
\end{align*}
and
\begin{align*} 
&\pd^2\big[ u(x)+\e^2 v(\frac{x}{\e};[u](\hx))+\frac\eta 2|x|^2+c|x-\hx|^2\big]_{x=\tx}=\\
=&D^2u(\tx)e_\a\cdot e_\a+ \pd^2_y v(\frac{\tx}{\e};[u](\hx))]+\eta +2c.
\end{align*}
Replacing  the previous identities in \eqref{eq:test_equation} and using  Lemma \ref{lem:Ham_eff}, Lemma \ref{lem:HJ_eff}, \eqref{eq:stima_hx} and $\tx\in  B(\hx,\e)\cap\G^\e$, we get
\begin{align*}
0\ge &u^\e(\tx)- \mu_\a  \big(D^2u(\tx)e_\a\cdot e_\a+ \pd^2_y v(\tx/\e;[u](\hx))\big)+H\Big(\tx,\frac{\tx}{\e}, Du(\tx)\cdot e_\a\Big)\\
-&M_2\big(\e \bC_2(1+2C_0)+\eta^{1/2}M_1+2c\e+\eta+2c\Big)\ge u^\e(\tx)\\
-& \mu_\a  \big(D^2u(\hx)e_\a\cdot e_\a+ \pd^2_y	v(\tx/\e;[u](\hx))\big)+H\Big(\hx,\frac{\tx}{\e}, Du(\hx)\cdot e_\a\Big)\\
-&M_2(\e^\d+\eta^{1/2})=	u^\e(\tx)+\bar H(\hx,Du(\hx),D^2u(\hx))-M_2(\e^\d+\eta^{1/2})\\
=&u^\e(\tx)-u(\hx)-M_2(\e^\d+\eta^{1/2})
\end{align*}
for some $M_2$, which may change from line to line but is always independent of $\e$ and $\g$; hence \eqref{eq:main_est}.
For every $x\in\G^\e$, by  $\tilde \phi(\tx)\geq \tilde\phi(\hx)=\phi(\hx)\geq \phi(x)$, we get
by \eqref{eq:main_est}, Lemma \ref{lem:HJ_eff} and Lemma \ref{lem:Ham_eff}
\begin{align*}
u^\e(x)-u(x)&\leq [u^\e(\tx)-u(\hx)]+[u(\hx)-u(\tx)]+\\
& \e^2\left[v(x/\e;[u](x))-v(\tx/\e;[u](\hx))\right]+\frac \eta 2 |x|^2\\
&\leq M_2\left[\e^\d +\eta^{1/2}\right]+C_0\e +2\bC_2(1+2C_0)\e^2+\frac \eta 2|x|^2.
\end{align*}
Letting $\eta\to 0^+$, we deduce
\begin{equation*}
u^\e(x)-u(x)\leq M_2\e^\d\qquad \forall x\in\G^\e.
\end{equation*}
Reversing the role of $u$ and $u^\e$, we get the statement.
\end{proof}

\appendix
\section{Appendix}

\begin{proof}[Proof of Lemma \ref{lem:Ham_eff}]
The proofs of existence and uniqueness of such a~$\rho$ and of relation~\eqref{eq:corr_prp1} rely on an easy adaptation of standard techniques; we refer the reader to~\cite{ab,al,evans}. Relation~\eqref{eq:corr_prp2} is due to~\eqref{eq:dc_infinito} and~\eqref{eq:corr_prp1}.


We prove an explicit formula for~$\rho$. Recall that, by the  assumptions, there are $N$ different Hamiltonians  $H_k$ and $N$ viscosity coefficients $\m_k$, $k=1,\dots,N$, corresponding to the vectors $e_k$. Integrating  the HJ equation in \eqref{eq:HJ_cell} along the $N$ arcs  $\G_k$ parallel to $e_k$ exiting from $\nu_i$ and denoting with $\mu_k$, $\g_k$   the corresponding coefficients in the Kirchhoff condition, we have
\begin{equation}\label{eq:calcolo_Hbar1}
0=\sum_{k=1}^N\g_k\int_{e_k}[-\mu_{k}\partial^{2}_{e_k}(v+Xy\cdot y/2)+H(x,y,\pd_{e_k}( P\cdot x))+\rho]dy.
\end{equation}
By periodicity of $v$, we have
\[
\int_{e_k} \partial^{2}_{e_k}v(y)dy=\partial_{e_k}v(1)-\partial_{e_k}v(0)=
-\partial_{-e_k}v(0)-\partial_{e_k}v(0).
\]
Replacing the previous identity in \eqref{eq:calcolo_Hbar1}, we have
\begin{align*}
0&=\sum_{k=1}^N\g_k \mu_{k}[\partial_{-e_k}v(0)+\partial_{e_k}v(0)]+\sum_{k=1}^N\g_k[ Xe_k\cdot e_k\\
&+\int_{e_k}H(x,y, P\cdot e_k) dy]+\rho\sum_{k=1}^N\g_k.
\end{align*}
Therefore, taking into account the Kirchhoff condition in \eqref{eq:HJ_cell} at $\n_i=0$ and observing that $\g_k=\g_{-k}$, where $\g_{-k}$ is the coefficient $\g_\a$ for the arc $-e_k$, we get~\eqref{eq:H_eff}.

\end{proof}
\begin{proof}[Proof of Lemma \ref{lem:HJ_eff}] 
The statement is obtained by standard arguments; for problem~\eqref{eq:HJ_hom},  we refer the reader to \cite[Theorem 15.5.1]{bc}, \cite{ls} and \cite{Morfe} while for problem~\eqref{eq:HJ_eff} we refer to \cite{ab,al}.
\end{proof}
\begin{proof}[Proof of Lemma \ref{lem:kirchhoff}]
Let $g:\R^N\to \R$ be  a smooth function. Since $\gamma_{i\alpha}=\g_\a$ and	$\g_\a$, $\mu_\a$ only depend on the direction $e_k$, parallel to $\G_\a$, we have
\begin{align*}
&\sum_{\alpha\in\mA_{i}}\gamma_{i\alpha}\mu_{\alpha}\partial_\a g(\nu_i)=Dg(\nu_i)\cdot \sum_{k=1}^N  (\g_k \m_k e_k+\g_k \m_k e_{-k}) \\
&=Dg(\nu_i)\cdot \sum_{k=1}^N  (\g_k\m_k e_k-\g_k\m_k e_{k})=0. 
\end{align*}
\end{proof}

\end{document}